\documentclass[12pt]{amsart}
\usepackage{latexsym,amssymb,amsmath}
\usepackage{tabularx}

\numberwithin{equation}{section}
\newtheorem{theorem}{Theorem}[section]
\newtheorem{lemma}[theorem]{Lemma}
\newtheorem{proposition}[theorem]{Proposition}
\newtheorem{corollary}[theorem]{Corollary}

\theoremstyle{definition}
\newtheorem{definition}[theorem]{Definition}

\newtheorem{remark}[theorem]{Remark}
\newtheorem{example}[theorem]{Example}

\newcommand{\NN}{\mathbb{N}}
\newcommand{\ZZ}{\mathbb{Z}}
\newcommand{\RR}{\mathbb{R}}

\newcommand{\LL}{\mathcal{L}}
\newcommand{\PP}{\mathbb{P}}

\newcommand{\TT}{\mathbb{T}}
\newcommand{\xx}{\underline{x}}

\renewcommand{\aa}{\mathbf{a}}
\newcommand{\bb}{\mathbf{b}}
\newcommand{\cc}{\mathbf{c}}

\newcommand{\uu}{\mathbf{u}}
\newcommand{\vv}{\mathbf{v}}

\newcommand{\ww}{\mathbf{w}}

\newcommand{\ev}{\operatorname{ev}}

\newcommand{\set}[1]{\left\{ #1 \right \} }

\newcommand{\ls}[1]{\left | #1 \right | }

\newcommand{\Span}[1]{\left < #1 \right > }
\newcommand{\rt}{\rightarrow}

\def\noqed{\renewcommand{\qedsymbol}{}}

\begin{document}

%%%%%%%%%%%%%%%%%%%%%%%%%%%%%%%%%%%%%%%%%%%%%%%%%%%%%%%%%%%%%%%%%%%%%

\title[Codes over a weighted torus]{Codes over a weighted torus}

\author{Eduardo Dias}
\address{Mathematics Institute, 
Zeeman Building, 
University of Warwick,
Coventry CV4 7AL, U.K.
}
\email{eduardo.m.dias@gmail.com}

\author{Jorge Neves}
\address{CMUC, Department of Mathematics, University of Coimbra
3001-454 Coimbra, Portugal.
}
\email{neves@mat.uc.pt}

\thanks{The second author was partially funded by CMUC, through
European program COMPETE/FEDER and FCT project PEst-C/MAT/UI0324/2011. Part of work was developed during a research visit of the second author 
to CINVESTAV of the IPN, M\'exico, under the financial support of a research grant from Santander Totta Bank (Portugal).}

% \keywords{Weighted projective Reed--Muller codes, numerical semigroups, vanishing ideals, minimum distance, index of regularity, Hilbert series.}

\begin{abstract} We define weighted projective Reed--Muller codes over a subset $X \subset \PP(w_1,\dots,w_s)$ in weighted projective space over a finite field. 
We focus on the case when the set $X$ is a projective weighted torus. We show that the vanishing ideal of $X$ is a lattice ideal and relate it with 
the lattice ideal of a minimal presentation of the semigroup algebra of the numerical semigroup $Q=\Span{w_1,\dots,w_s}\subset \NN$.  We compute the index of regularity of 
the vanishing ideal as function of the weights of the projective space and the Frobenius number of $Q$. We compute the 
basic parameters of weighted projective Reed--Muller codes over a $1$-dimensional weighted torus and prove they are maximum distance separable codes.
\end{abstract}

\maketitle

%%%%%%%%%%%%
%%%%%%%%%%%%
%%%%%%%%% SECTION
%%%%%%%%%%%%
%%%%%%%%%%%%

\section{Introduction}

A standard projective Reed--Muller code, $C_X(d)$, is the image of the degree $d$ homogeneous component of a standard polynomial ring $K[t_1,\dots,t_s]$ over a finite
field $K$ by a homomorphism defined by evaluation of forms of degree $d$ on the points of and arbitrary subset $X\subset \PP^s$. 
In this work we define the notion of \emph{weighted projective Reed--Muller code} (see Definition~\ref{def: weighted projective Reed--Muller codes}). This notion
differs from the standard definition in that the grading of $K[t_1,\dots,t_s]$, which is given by $\deg(t_i)=w_i\geq 1$, for coprime integers $w_i$, 
is not necessarily the standard one. We focus on the family of codes $C_\TT(d)$ associated to 
a weighted $(s-1)$-dimensional projective torus $\TT(w_1,\dots,w_s)$ (see Definition~\ref{def: weighted projective torus}).

Standard projective Reed--Muller codes of order $d\leq q$ were defined and studied by Lachaud in \cite {La88, La90} and, for all $d\geq 0$,  by S{\o}rensen in \cite{So}. 
Much of the recent research on projective Reed--Muller codes over an arbitrary subset of $X\subset \PP^s$ focuses on the computation of their basic
para\-meters: \emph{length}, \emph{dimension} and \emph{minimum distance} (see Definition~\ref{def: basic parameters}). 
When $X=\PP^s$ all the basic parameters are known (\emph{cf.}~\cite{La90,So});
in particular, projective Reed--Muller codes over $\PP^1$ are \emph{maximum distance separable} codes. In general, 
an approach to this computation using commutative algebra (as in \cite{ReTa97}) 
relies on a good understanding of $I_X\subset K[t_1,\dots,t_s]$, the vanishing ideal of $X$. 
Many authors have studied projective Reed--Muller codes over a subset $X\subset \PP^s$ for which the ideal $I_X$ is well understood; \emph{e.g.}, when
$X$ is the set of rational points of a complete intersection, \emph{cf.}~\cite{BaFo06,Co12,CoDu13,DuReTa,GoLiSc,Han,SaVPVi12,Sop}, when $X$ is the Segre embedding of the product of two projective spaces,
\emph{cf.}~\cite{GoReTa02}, when $X$ is the Veronese variety \emph{cf.}~\cite{ReTa00}, when $X$ is an affine cartesian product \emph{cf.}~\cite{Ca13,LoReVi12}, or when $X$ is the projective torus in $\PP^s$, or an \emph{algebraic toric subset}, 
\emph{cf.}~\cite{NeVPVi, NeVPVi11, algcodes, SaVPVi11, SaVPVi12}. 

The advantage of working with subsets of the torus is that for a certain subclass of these subsets 
(consisting of algebraic toric subsets, as defined by Villarreal \emph{et al.} in \cite{ReyViZa00,algcodes}) the ideal $I_X$ is a lattice ideal.
Like in the standard case, $I_\TT$, the vanishing ideal of the weighted torus $\TT(w_1,\dots,w_s)$, is also a lattice ideal. 
Indeed, we show that $I_\TT$ is Cohen--Macaulay, $1$-dimensional and can be obtained from the lattice ideal of a minimal presentation of the semigroup algebra of the numerical semigroup $Q=\Span{w_1,\dots,w_s}\subset \NN$
(\emph{cf.}~Theorem~\ref{thm: relation between I_T and the ideal of the minimal presentation}). The lattice ideal of a 
minimal presentation of the semigroup algebra was first studied by Herzog in \cite{herzog}. He gives a sufficient condition for this ideal to
be a complete intersection (see Remark~\ref{rmk: when do we get a complete intersection}), which, combined with our results, 
is also a sufficient condition for $I_\TT$ to be a complete intersection. The relation between $I_\TT$ and the lattice ideal of a minimal presentation
of $Q$ enables the computation of the Hilbert Series and the index of regularity of $K[t_1,\dots,t_s]/I_\TT$ in terms of $w_1,\dots,w_s$ and 
the Frobenius number of $Q$ (\emph{cf.}~Theorem~\ref{thm: Hilbert Series} and Corollary~\ref{cor: regularity}). The importance, from a coding theory point of view,  
of the knowledge of the index of regularity is clearer in the 
case of standard projective Reed--Muller codes. Here, the function $\dim_K C_X(d)$ 
is strictly increasing and the value of $d$ for which $\dim_K C_X(d)$ becomes constant 
and equal to the dimension of the ambient space (thus, for which $C_X(d)$ becomes the trivial code) is precisely given by the index of regularity. In the weighted case $\dim_K C_X(d)$ is not necessarily
an increasing function and we may get some trivial codes $C_X(d)$ before $d$ reaches the index of regularity (\emph{cf.}~Example~\ref{example}). However for $d$ greater than or equal to 
index of regularity $C_X(d)$ is always a trivial code. 

The structure of the article is as follows. In Section~\ref{sec: vanishing} we study the vanishing ideal of a weighted projective torus. The basic definitions are
recalled. We show that $I_\TT$ is a $1$-dimensional, Cohen-Macaulay lattice ideal and relate it to the lattice ideal of a minimal presentation of the semigroup algebra of 
\mbox{$Q=\Span{w_1,\dots,w_s}$}. In Section~\ref{sec: weighted PRM codes} we define weighted projective Reed-Muller codes. We compute the length of the weighted projective 
Reed--Muller codes over the weighted torus  (\emph{cf.} Proposition~\ref{prop: length of the weighted torus}) and we compute the Hilbert Series and the index of regularity of $K[t_1,\dots,t_s]/I_\TT$
(\emph{cf}.~Theorem~\ref{thm: Hilbert Series} and Corollary~\ref{cor: regularity}). In Section~\ref{sec: codes in 1-dim} we study projective Reed--Muller codes over a $1$-dimensional 
weighted torus $\TT(w_1,w_2)$. We compute their dimensions and minimum distances (\emph{cf.} Proposition~\ref{prop: dimension of the code over the torus} 
and Theorem~\ref{thm: computing the minimum distance}) and we show they are \emph{maximum distance separable} codes.

\section{Vanishing ideal of a weighted torus}\label{sec: vanishing}

Let $K$ be a finite field with $q$ elements. We denote by $K^*$ the cyclic group of invertible elements of $K$. 
Given two $s$-tuples $\vv=(n_1,\dots,n_s)\in\NN^s$ and $u=(u_1,\dots,u_s)\in R^s$, where $R$ is a commutative ring with identity, 
$u^\vv$ denotes $u_1^{n_1}\cdots u_s^{n_s}\in R$. We will use this notation for vectors of variables $t=(t_1,\dots,t_s)$ or 
$y=(y_1,\dots,y_s)$. We reserve boldface for vectors of integers. 

\begin{definition}
An ideal $I\subset K[t_1,\dots,t_s]$ generated by \emph{binomials}, \emph{i.e.}~polynomials
of the form $\alpha t^\aa - \beta t^\bb$, for some $\alpha, \beta \in K$ and $\aa,\bb \in \NN^s$, is called a \emph{binomial ideal}.
\end{definition}

We will only deal with \emph{pure binomials}, \emph{i.e.}, binomials for which $\alpha= \beta =1$.

\begin{definition}
Let $\LL \subset \ZZ^s$ be a lattice. The lattice ideal $I_\LL\subset K[t_1,\dots,t_s]$ is the ideal generated by $t^\aa-t^\bb$ for all $\aa,\bb\in \NN^s$ such that $\aa-\bb \in \LL$.
\end{definition}

\begin{definition}
Let $w_1,\dots,w_s$ be a set of positive integers satisfying \mbox{$\gcd(w_1,\dots,w_s)=1$}. 
The \emph{weighted projective space}, that we denote by $\PP(w_1,\dots,w_s)$, is the quotient \mbox{$(K^{s}\setminus 0)/K^*$}, where $\lambda\in K^*$ acts by 
$\lambda (x_1,\dots,x_s) = (\lambda^{w_1}x_1,\dots, \lambda^{w_s}x_s)$.
\end{definition}

Let $K[t_1,\dots,t_s]$ be the coordinate ring of $\PP(w_1,\dots,w_s)$, endowed with the grading given by $\deg(t_i)=w_i$, for all $1\leq i\leq s$. 
Set $\ww = (w_1,\dots,w_s)$. We remark that a binomial $t^\aa-t^\bb$, with $\aa,\bb\in \NN^s$ is 
homogeneous if and only if $\Span{\aa-\bb,\ww}=0$.

\begin{definition}\label{def: weighted projective torus}
The \emph{weighted projective torus}, $\TT(w_1,\dots,w_s)\subset \PP(w_1,\dots,w_s)$ is the set:
\[
\TT(w_1,\dots,w_s):=\set{(x_1,\dots,x_s)\in \PP(w_1,\dots,w_x) : x_1\cdots x_s\not = 0}. 
\]
\end{definition}

\begin{definition} For a set $X\subset \PP(w_1,\dots,w_s)$ the ideal of $K[t_1,\dots,t_s]$ generated by all homogeneous polynomials that vanish on
$X$ is called the \emph{vanishing ideal} of $X$ and is denoted by $I_X$. We denote the vanishing ideal of $\TT(w_1,\dots,w_s)$ by $I_\TT$. 
\end{definition}

Over an infinite field $I_\TT = (0)$. However given that $K$ has $q$ elements and its multiplicative group is cyclic of order $q-1$,
we get $$\bigl(t_2^{w_1(q-1)}-t_1^{w_2(q-1)},\dots,t_s^{w_1(q-1)}-t_1^{w_s(q-1)}\bigr)\subset I_\TT.$$
For general $w_1,\dots,w_s$ this inclusion is strict. The precise structure of a minimal generating set of $I_\TT$ is closely related with 
the numerical semigroup $Q=\Span{w_1,\dots,w_s}\subset \NN$, \emph{cf.}~Remark~\ref{rmk: when do we get a complete intersection}.

The proof of the next lemma follows closely that of Theorem~2.1 in \cite{algcodes}.

\begin{lemma}\label{lemma: I is generated by pure binomials}
In the polynomial ring extension $K[t_1,\dots,t_s]\subset K[t_1,\dots,t_s,y_1,\dots,y_s,z]$,
\begin{equation}\label{eq: basic equality}
I_\TT = (\set{t_i-y_iz^{w_i}}_{i=1}^s\cup \{y_i^{q-1}-1\}_{i=1}^s)\cap K[t_1,\dots,t_s].
\end{equation}
In particular, $I_\TT$ is generated by homogeneous pure binomials. 
\end{lemma}

\begin{proof}
Let $J$ denote the ideal on the right hand side of (\ref{eq: basic equality}). We start by showing that $I_\TT\subset J$. Let $f=\sum_{j=1}^r \alpha_j t^{\vv_j}$, for some 
$\alpha_j\in K$ and $\vv_j\in \NN^s$ be a homogeneous polynomial of degree $d$. Let $\uu=(u_1,\dots,u_s)\in\NN^s$, then
\[
t^{\uu} = t_1^{u_1}\cdots t_s^{u_s} = (t_1-y_1z^{w_1} +y_1z^{w_1})^{u_1}\cdots (t_s-y_sz^{w_s} +y_sz^{w_s})^{u_s} = \sum_{i=1}^s (t_i-y_iz^{w_i}) g_{\uu,i} + z^{d} y^\uu,
\]
where $g_{\uu,j}\in K[t_1,\dots,t_s,y_1,\dots,y_s,z]$ and $d=\sum_{i=1}^s u_iw_i$. Notice that since $f$ is homogeneous of degree $d$, this sum with $\uu$ replaced 
by $\vv_j$ equals $d$. Expanding each $t^{\vv_j}$ in this way, we get
\[
f = \sum_{i=1}^s (t_i-y_i z^{w_i})g_i + z^d \sum_{j=1}^r\alpha_j y^{\vv_j},
\]
where $g_i=\sum_{j=1}^r g_{\vv_j,i}$. Dividing the polynomial $\sum_{j=1}^r\alpha_j y^{\vv_j}$ by the set $\{y_i^{q-1}-1\}_{i=1}^s$ in the polynomial ring $K[y_1,\dots,y_s]$,
we deduce that there exist \mbox{$h_i,g\in K[y_1,\dots,y_s]$}, with $g$ of degree $<q-1$ in each variable such that 
\begin{equation}\label{eq: lemma first decomposition}
f = \sum_{i=1}^s (t_i-y_i z^{w_i})g_i + z^d \sum_{i=1}^s (y_i^{q-1}-1)h_i + z^d g. 
\end{equation}
Let us show that $g(x_1,\dots,x_s) =0$ for all $(x_1,\dots,x_s)\in (K^*)^s$, since, by the Combinatorial Nullstellensatz, this implies that $g=0$.
Regarding $(x_1,\dots,x_s)\in (K^*)^s$ as a system of homogeneous coordinates of a point in $\TT(w_1,\dots,w_s)$ we get, by assumption, $f(x_1,\dots,x_s)=0$. Hence,
setting $t_i=y_i=x_i$ and $z=1$ in (\ref{eq: lemma first decomposition}):
\[
0 = \sum_{i=1}^s (x_i^{q-1}-1)h_i(x_1,\dots,x_s) + g(x_1,\dots,x_s) = g(x_1,\dots,x_s).
\]
To show the reverse inclusion we start by remarking that $J$, being an elimination of an ideal generated by pure binomials, is itself generated by pure binomials (\emph{cf.}~\cite{EiSt96}). 
It suffices to show that any pure binomial in $J$ is also in $I_\TT$. In passing we will show that such a binomial is homogeneous. This will imply the
assertion that $I_\TT$ is generated by homogeneous pure binomials. Let $t^\aa - t^\bb\in J$, for some $\aa,\bb \in \NN^s$. Then there exist $g_i,h_i\in K[t_1,\dots,t_s,y_1,\dots,y_s,z]$ such that 
\begin{equation}\label{eq: expressing binomial}
t^\aa - t^\bb = \sum_{i=1}^s (t_i-y_i z^{w_i})g_i + \sum_{i=1}^s (y_i^{q-1}-1)h_i.
\end{equation}
Substituting in (\ref{eq: expressing binomial}) $1$ for $y_i$ and $z^{w_i}$ for $t_i$ we get $z^{a_1w_1}\cdots z^{a_sw_s}-z^{b_1w_1}\cdots z^{b_sw_s}=0$ and therefore
$\Span{\aa,\ww}=\Span{\bb,\ww}$, \emph{i.e.}, $t^\aa-t^\bb$ is homogeneous. Finally, to show that it vanishes on an arbitrary point $(x_1,\dots,x_s)$ of the weighted torus, 
we use (\ref{eq: expressing binomial}) but this time substituting $t_i$ and $y_i$ by $x_i$ and $z$ by $1$.
\end{proof}

\begin{remark}
More generally, it can be shown that Lemma~\ref{lemma: I is generated by pure binomials} holds for the 
vanishing ideal of a \emph{weighted toric subset} $X\subset \TT(w_1,\dots,w_s)$ parameterized by $\vv_1,\dots,\vv_s \in \NN^n$, for some integer $n>0$. More precisely,
by analogy with the standard case (\emph{cf.}~\cite[\S 2]{algcodes}), when $X$ the subset defined by:
\[
X = \set{(\xx^{\vv_1},\dots,\xx^{\vv_s})\in \PP(w_1,\dots,w_s) : \xx \in (K^*)^n},
\]
where $w_i = \frac{1}{d}\sum_{j=1}^n (\vv_i)_j$, with
$d = \gcd \{\sum_{j=1}^n (\vv_1)_j,\dots,\sum_{j=1}^n (\vv_s)_j\}$. In the case when $\vv_1,\dots, \vv_s$ coincide with 
the incidence vectors of a uniform clutter (and, in particular, of a graph), this notion 
coincides with the notion of toric subset parameterized by $\vv_1,\dots, \vv_s$, as defined in \cite{algcodes, SaVPVi11}.
\end{remark}

We denote by $\ww^\perp$  the orthogonal in $\RR^s$ of $\Span{\ww}$, with respect to the canonical inner product.

\begin{theorem}\label{thm: relation between I_T and the ideal of the minimal presentation}
Let  $\LL=(q-1)\bigl(\ww^\perp \cap \ZZ^s\bigr )$. Then $I_\TT=I_\LL$ and, in particular, $I_\TT$ is $1$-dimensional and Cohen--Macaulay.
\end{theorem}

\begin{proof}
To prove the equality we need to show the inclusion $I_\TT \subset I_\LL$.
Let $t^\aa-t^\bb\in K[t_1,\dots,t_s]$ be a homogeneous binomial vanishing on $\TT(w_1,\dots,w_s)$. Fix $i\in \set{1,\dots,s}$ and let $\alpha$ be a generator of $K^*$. 
Consider $(1,\dots,1,\alpha,1,\dots,1)\in \TT(w_1,\dots,w_s)$, with $\alpha$ in the $i$-th position. Evaluating $t^\aa-t^\bb$ at this point we obtain
$$
\alpha^{a_i}-\alpha^{b_i} = 0 \iff \alpha^{a_i-b_i} = 1 \iff a_i-b_i \equiv 0 \pmod{q-1}. 
$$
Since this holds for any $i\in \set{1,\dots,s}$ we obtain $a-b\in (q-1)\NN^s$. As $t^\aa-t^\bb$ is homogeneous we
get $\Span{\aa-\bb,\ww}=0$ and hence $\aa-\bb\in \LL$. This proves that $I_\TT = I_\LL$ 
Since the rank of $\LL$ is $s-1$ we deduce that $I_\TT$ is $1$-dimensional (\emph{cf}.~\cite[Proposition~7.5]{MiStu05}). Furthermore, since 
\begin{equation}\label{eq: fat inclusion}
\bigl(t_2^{w_1(q-1)}-t_1^{w_2(q-1)},\dots,t_s^{w_1(q-1)}-t_1^{w_s(q-1)}\bigr)\subset I_\TT 
\end{equation}
and consequentially
\mbox{$V(I_\TT,t_i)=\set{0}$},
by \cite[Proposition~5.3]{OCPlVi}, $I_\TT$ is Cohen--Macaulay.
\end{proof}

\begin{definition}
Let $Q\subset \NN$ denote the submonoid of $\NN$ generated by $w_1,\dots,w_s\in \NN$. 
Since \mbox{$\gcd(w_1,\dots,w_s)=1$}, $Q$ is a \emph{numerical semigroup}, \emph{i.e.}, it has finite complement. 
The semigroup algebra, denoted by $K[Q]$ is the subalgebra of the polynomial ring $K[z]$	
given by $K[z^{w_1},\dots,z^{w_s}]$.
\end{definition}

\begin{lemma}\label{lemma: relation of lattice ideal with the semigroup algebra}
Let $\LL^\flat=\ww^\perp \cap \ZZ^s$ and consider $I_{\LL^\flat}\subset K[t_1,\dots,t_s]$ the corresponding lattice ideal.
Then $I_{\LL^\flat}$ is a homogeneous ideal
and \mbox{$K[t_1,\dots,t_s]/I_{\LL^\flat}\simeq K[Q]$}.
\end{lemma}

\begin{proof}
See \cite[Theorem~7.3]{MiStu05} and also \cite[Proposition 1.4]{herzog}.
\end{proof}

The following lemma yields a relation between $I_\TT$ and $I_{\LL^\flat}$ is given by the following lemma.

\begin{lemma}\label{lemma: relation between the 2 lattice ideals}
Let $\LL^\flat = \ww^\perp \cap \ZZ^s$. Suppose that $I_{\LL^\flat} = (t^{\aa_1}-t^{\bb_1},\dots,t^{\aa_r}-t^{\bb_r})$, for some 
$\aa_i,\bb_i\in\NN^s$. Then $I_\TT = (t^{(q-1)\aa_1}-t^{(q-1)\bb_1},\dots,t^{(q-1)\aa_r}-t^{(q-1)\bb_r})$. 
\end{lemma}
\begin{proof}
Let $J=(t^{(q-1)\aa_1}-t^{(q-1)\bb_1},\dots,t^{(q-1)\aa_r}-t^{(q-1)\bb_r})$. Since $(q-1)\aa_i-(q-1)\bb_i \in \LL$, the inclusion $J\subset I_\TT$ is 
clear. Conversely, let $t^\aa-t^\bb\in I_\TT$. Then there exist $\cc^+,\cc^- \in \NN^s$ such that $\cc^+-\cc^- \in \LL^\flat$ and $\aa-\bb = (q-1)(\cc^+-\cc^-)$.
Since $t^{\cc^+}-t^{\cc^-}\in I_{\LL^\flat}$ there exist $h_i\in K[t_1,\dots,t_s]$ such that
$t^{\cc^+}-t^{\cc^-} = \sum_{j=1}^r (t^{\aa_j}-t^{\bb_j})h_j$. Substituting in this equality $t_i^{q-1}$ for $t_i$, for every $i=1,\dots,s$, we deduce that $t^\aa-t^\bb \in J$. Hence $I_\TT\subset J$.
\end{proof}

\begin{remark}\label{rmk: when do we get a complete intersection}
In \cite{herzog}, Herzog shows that if the condition
\begin{equation}\label{eq: sufficient condition for complete intersection}
 \operatorname{lcm}(\gcd\set{w_1,\dots,w_{i-1}},w_{i})\in \Span{w_1,\dots,w_{i-1}} 
\end{equation}
is satisfied, for every $i=2,\dots,s$, then $I_{\LL^\flat}$ is generated by the binomials 
$t_{i}^{c_{i}}-\prod_{j=1}^{i-1} t_j^{r_{ij}}$,  where,  for \mbox{$i=2,\dots,s$,} \mbox{$c_i=\gcd\set{w_1,\dots,w_{i-1}}/\gcd\set{w_1,\dots,w_{i}}$},
and $r_{ij}\in \NN$ are nonnegative integers such that 
\mbox{$\operatorname{lcm}(\gcd\set{w_1,\dots,w_{i-1}},w_{i})=c_{i}w_{i}=\sum_{j=1}^{i-1} r_{ij}w_j$}. In particular, in this situation,
$I_{\LL^\flat}$ is a complete intersection and, by Lemma~\ref{lemma: relation between the 2 lattice ideals}, $I_\TT$ is also a complete intersection. 
It can be checked that (\ref{eq: sufficient condition for complete intersection}) is satisfied for $s=2$ or 
if $w_1=1$ or $2$. The case when $w_i=1$, for all $i=1,\dots,s$, is worth highlighting, for these $w_i$ (\ref{eq: fat inclusion}) is an equality.
If $s=3$ then (\ref{eq: sufficient condition for complete intersection}) is also a necessary condition for $I_{\LL^\flat}$ to 
be a complete intersection (\emph{cf.}~\cite[Theorem~3.10]{herzog}). This condition is no longer necessary when $s=4$; it can 
be shown (\emph{cf.}~\cite[Example~3.9]{FischerShapiro}) that if $\ww = (20,30,33,44)$ then $I_{\LL^\flat}$ is a complete intersection, 
despite the fact that no ordering of $20,30,22,44$ 
satisfies (\ref{eq: sufficient condition for complete intersection}). 
A recursive method for deciding whether, for a given $\ww$, the ideal $I_{\LL^\flat}$ is a complete intersection is given by Delorme in \cite{delorme}.
\end{remark}

%%%%%%%%%%%%%%%%%%%%
%%%%%%%%%%%%%%%%%%%%
%%%%%%%%%%%%%%%%%%%%
%%%%%%%%%%%%%%%%%%%%
%%%%%%%%%%%%%%%%%%%%
%%%%%%%%%%%%%%%%%%%%
%%%%%%%%%%%%%%%%%%%%
%%%%%%%%%%%%%%%%%%%%
%%%%%%%%%%%%%%%%%%%%
%%%%%%%%%%%%%%%%%%%%
%%%%%%%%%%%%%%%%%%%%

\section{Weighted projective Reed--Muller codes}\label{sec: weighted PRM codes}

\begin{definition}\label{def: weighted projective Reed--Muller codes}
Let $X\subset \PP(w_1,\dots,w_s)$ and set $m=|X|$. Fix $\xx_1,\dots,\xx_m\in K^s$ 
systems of homogeneous coordinates for the points of $X$. Given $d\geq 0$, 
let \mbox{$\ev_d \colon K[t_1,\dots,t_s]_d \rt K^m$} be the map defined by $f\mapsto (f(\xx_1),\dots,f(\xx_m))$ for all $f\in K[t_1,\dots,t_s]_d$. The image of $\ev_d$, denoted by $C_X(d)$, 
is called \emph{the weighted projective Reed--Muller code over $X$} (or, if the context is clear, simply the code over $X$) \emph{of order $d$}.
\end{definition}

\begin{remark}
Let $\xx'_1,\dots,\xx'_m$ be a different choice of homogeneous coordinates of the points of $X$. Denote by $\ev'_d$ the corresponding evaluation map. Then,
there exist $\lambda_1,\dots,\lambda_m \in K^*$ such that $\ev'_d(K[t_1,\dots,t_s]_d)$ is the image of $\ev_d(K[t_1,\dots,t_s]_d)$ by the linear map defined by 
$(y_1,\dots,y_m) \mapsto (\lambda_1 y_1,\dots,\lambda_m y_m)$, for every $(y_1,\dots,y_m)\in K^m$. \emph{I.e.}, the $2$ codes are equivalent. 
\end{remark}

\begin{definition}\label{def: basic parameters}
The basic parameters of a linear code $0\not = C\subset K^m$ are the \emph{length}, the \emph{dimension} and the \emph{minimum distance}. The length is $m$, the dimension of the
ambient vector space; the dimension is its dimension as a vector space and the minimum distance, that is denoted by 
$\delta_C$, is given by $\min \set{\|\xx\| : \xx\in C\setminus 0}$ where $\|\xx\|$ is the \emph{Hamming weight} 
of $\xx$, \emph{i.e.}, the number of nonzero components of $\xx$. A code is said \emph{maximum distance separable} if the \emph{singleton bound}:
\[
\delta_C \leq \operatorname{length}(C) -\dim_K C +1,  
\]
which is always satisfied for a linear code, is an equality.
\end{definition}

\begin{remark}
If $C=C_X(d)$ is a weighted projective Reed--Muller code over $X\subset \PP(w_1,\dots,w_s)$ of order $d$, the length is equal to $|X|$ 
(thus is independent of $d$) and the minimum distance can be computed as $m$ minus the maximum number of zeros a homogeneous 
polynomial of degree $d$ can attain on $X$ without belonging to $I_X$. 
Maximum distance separable codes are codes that, for their length and dimension, maximize minimum distance, in other 
words, have maximum error-correcting capability.
\end{remark}

In this work we shall focus on the codes over $X=\TT(w_1,\dots,w_s)$. We abbreviate the notation for the codes to $C_\TT(d)$ and for their minimum distance to
$\delta_\TT(d)$.

\begin{proposition}\label{prop: length of the weighted torus}
The length of $C_\TT(d)$ is $(q-1)^{s-1}$. 
\end{proposition}

\begin{proof}
The length of $C_\TT(d)$ coincides with the cardinality of the set of $K$-points of $X$. 
Since this set can be seen as the quotient $(K^*)^s/K^*$ by the induced action of $K^*$,  
all we need to check is that the orbits have cardinality $q-1$. Assume that 
\[
 \lambda(x_1,\dots,x_s) = \mu (x_1,\dots,x_s)
\]
for some $\lambda,\mu \in K^*$. Then $\lambda^{w_i}x_i = \mu^{w_i}x_i \iff (\lambda/\mu)^{w_i}=1$, hence
$\operatorname{ord}(\lambda/\mu)$ divides $w_i$, for all $i$. Since, by 
assumption, $\gcd\set{a_1,\dots,a_s}=1$ we deduce that $\lambda = \mu$.
\end{proof}

Let $M$ be a finitely generated graded $K[t_1,\dots,t_s]$-module. The Hilbert function of $M$ is the function $\varphi_M\colon \ZZ \rt \ZZ$ defined by 
$\varphi_M(d) = \dim_K M_d$. This function is quasi-polynomial of degree $\dim M-1$, \emph{i.e.}, there exist a
positive integer $g$ (the period) and $P_0,\dots,P_{g-1}$, polynomials of same leading term and of degree 
$\dim M -1$, such that, for $d\gg 0$, $\varphi_M(d)=P_i(d)$, where $d\equiv i \pmod{g}$ 
(\emph{cf}.~Serre's Theorem \cite[Theorem~4.4.3]{BH}).

\begin{definition} The index of regularity of $M$ is the least $r\geq 0$ such that for $d\geq r$, 
\[
 d\equiv i \pmod{g} \implies \varphi_M(d)=P_i(d).
\]
\end{definition}

The Hilbert series of $M$ is given by $H_M(t)=\sum_{d=0}^\infty \varphi_M(d)t^d$. $H_M(t)$ is a rational function 
(\emph{cf.}~\cite[Proposition 4.1.3]{Villa} or \cite[Proposition 4.4.1]{BH}) and its degree is called the $a$-invariant of $M$. By Serre's theorem the 
index of regularity of $M$ is equal to $\deg H_M(t) + 1$. 
\smallskip 

When $M= K[t_1,\dots,t_s]/I_\TT$, we abbreviate the notation for its Hilbert function and Hilbert series
to $\varphi_\TT$ and $H_\TT(T)$, respectively. We remark that $\dim_K C_\TT(d)=\varphi_\TT(d)$. Further, 
since $I_\TT$ is $1$-dimensional, $\varphi_\TT$ becomes constant, equal to $m=(q-1)^{s-1}$, the number of points of $\TT(w_1,\dots,w_s)$, for $d$ greater than 
or equal to the index of regularity of $K[t_1,\dots,t_s]/I_\TT$. For such $d$, 
$C_X(d)$ is a trivial code.

\begin{remark}
If one of the weights is equal to $1$, say $w_1=1$, then multiplication by $t_1$ induces a 
monomorphism $C_\TT(d)\hookrightarrow C_X(d+1)$. Moreover, one can also check that $\delta_\TT(d)\geq \delta_\TT(d+1)$. 
In the standard case (\emph{i.e.}~when
all $w_i=1$), for $d$ up to the index of regularity minus $1$, the inclusion is proper and the inequality is strict, implying that 
$\dim_K C_X(d)$ is strictly increasing and $\delta_{C_X(d)}$ is strictly decreasing.
\end{remark}

Let $G:=\NN\setminus Q$ denote the set of gaps of the numerical semigoup $Q$ and denote by $\mathsf{g}_Q:=\max G$ the Frobenius number of $Q$.

\begin{theorem}\label{thm: Hilbert Series}
The Hilbert series of $K[t_1,\dots,t_s]/I_\TT$ is given by 
\begin{equation}\label{eq: Hilbert Series}
H_{\TT}(t)=\frac{\bigl(\frac{1}{1-t^{q-1}}-\sum_{a\in G} t^{a(q-1)}\bigr)\prod_{i=1}^s (1-t^{w_i(q-1)})}{\prod_{i=1}^s (1-t^{w_i})}\cdot 
\end{equation}
\end{theorem}

\begin{proof}
Let $M=K[t_1,\dots,t_s]/I_{\LL^\flat}$. By Lemma~\ref{lemma: relation of lattice ideal with the semigroup algebra}, $M\simeq K[Q]$, hence
\begin{equation}\label{eq: first Hilbert Series}
\textstyle H_{M} = \sum_{a\in Q} t^a = \frac{1}{1-t} - \sum_{a\in G} t^a =\displaystyle 
\frac{\bigl( \frac{1}{1-t} - \sum_{a\in G} t^a \bigr)\prod_{i=1}^s (1-t^{w_i})}{\prod_{i=1}^s (1-t^{w_i})}\cdot 
 \end{equation}
Notice that, since $(1-t^{w_1})/(1-t)$ is a polynomial, the numerator of (\ref{eq: first Hilbert Series}) is a polynomial. 
Let \mbox{$I_{\LL^\flat} = (t^{\aa_1}-t^{\bb_1},\dots,t^{\aa_r}-t^{\bb_r})$}, for some 
$\aa_i,\bb_i\in\NN^s$. Then, by Lemma~\ref{lemma: relation between the 2 lattice ideals}, $$I_\TT = (t^{(q-1)\aa_1}-t^{(q-1)\bb_1},\dots,t^{(q-1)\aa_r}-t^{(q-1)\bb_r})$$
and by \cite[Lemma~3.7]{NeVPVi} we get (\ref{eq: Hilbert Series}).
\end{proof}

\begin{corollary}\label{cor: regularity}
The index of regularity of $K[t_1,\dots,t_s]/I_\TT$ is $(q-2)\bigl(\sum_{i=1}^s  w_i + \mathsf{g}_Q \bigr ) +\mathsf{g}_Q+1$.
\end{corollary}

\begin{example} \label{example}
Suppose $K=\operatorname{GF}(4)$, $X = \TT(3,4,5)$ and consider the corresponding family of codes $C_\TT(d)$. By Proposition~\ref{prop: length of the weighted torus} 
these are codes of length $9$. Using \cite{M2}, we can check that the ideal $I_\TT$ is minimally generated by the binomials $t_2^6+t_1^3t_3^3$, $t_1^9+t_2^3t_3^3$, $t_1^6t_2^3+t_3^6$ and thus 
is not a complete intersection. From Theorem~\ref{thm: Hilbert Series}, 
\[
H_\TT(t)= \frac{ 1-t^{24}-t^{27}-t^{30}+t^{39}+t^{42}}{(1-t^5)(1-t^4)(1-t^3)}\cdot
\]
Hence the index of regularity of $K[t_1,\dots,t_s]/I_\TT$ is $42-12+1=31$. This number can also be computed using Corollary~\ref{cor: regularity}.

\begin{table}[h]\caption{Parameters of $C_\TT(d)$, with $\ww=(3,4,5)$ and $K=\operatorname{GF}(4)$}\label{table: example of the codes over a w torus} \label{table: a code over GF(4)}
\newcolumntype{x}[1]{>{\hspace{0pt}}p{#1}}
\renewcommand{\arraystretch}{1.3}
\vspace{-.3cm}
\begin{flushleft}
\begin{tabular}[h]{x{.8cm}x{.45cm}x{.45cm}x{.45cm}x{.45cm}x{.45cm}x{.45cm}x{.45cm}x{.45cm}x{.45cm}x{.45cm}x{.45cm}x{.45cm}x{.45cm}x{.45cm}}
 $d$       & 0 & 1  & 2  & 3 & 4   &  5 & 6  &  7 &  8 &  9 & 10 & 11 & 12 & 13 	 \tabularnewline\noalign{\hrule height 1pt} 
 $\dim$    & 1 & 0  & 0  & 1 & 1   &  1 & 1  &  1 &  2 &  2 & 2  &  2 &  3 & 3   	 \tabularnewline\noalign{\hrule height 0.25pt} 
 $\delta$  & 9 & -- & -- & 9 & 9   &  9 & 9  &  9 &  6 &  6 & 6  &  6 &  6 & 6           \tabularnewline\noalign{\hrule height 1pt}
\end{tabular}
\end{flushleft}
\vspace{.25cm}
\begin {flushright}
\begin{tabular}[h]{x{.45cm}x{.45cm}x{.45cm}x{.45cm}x{.45cm}x{.45cm}x{.45cm}x{.45cm}x{.45cm}x{.45cm}x{.45cm}x{.45cm}x{.45cm}x{.45cm}x{.45cm}x{.45cm}x{.45cm}}
  14 & 15  & 16  & 17 & 18 & 19 & 20 & 21 & 22 & 23 & 24 & 25 & 26 & 27 & 28 & 29 & 30 	 \tabularnewline\noalign{\hrule height 1pt} 
   3 &  4  &  4  &  4 &  5 &  5 & 6  &  6 &  6 &  7 &  7 &  8 &  8 &  7 &  9 & 9  & 8 	 \tabularnewline\noalign{\hrule height 0.25pt} 
   6 &  3  &  3  &  4 &  3 &  3 & 3  &  3 &  3 &  2 &  2 &  2 &  2 &  2 &  1 & 1  & 2     \tabularnewline\noalign{\hrule height 1pt}
\end{tabular}
\end{flushright}
\end{table}
Table~\ref{table: a code over GF(4)} shows the dimension and minimum distance of $C_\TT(d)$, for $d=0,\dots,30$, computed using \cite{M2}. 
One feature to bear in mind is that, unlike standard projective Reed--Muller codes, $\dim_K C_\TT(d)$ is not strictly increasing and 
$\delta_\TT(d)$ is not strictly decreasing. Nevertheless, this family of codes is not necessarily redundant. For example, the two
$4$-dimensional codes with equal minimum distance ($d=15$ and $16$) are not equivalent. Indeed, these codes have generating matrices in standard form $(I_4|A)$ and
$(I_4|B)$ where $A,B\in M_{4\times 5}\operatorname{GF}(4)$ are given by:
\[
\renewcommand{\arraystretch}{.75}
\left (
\begin{array}{ccccc}
1& 1& 0& 0& 0\\ 
0& 0& 1& 1& 0\\ 
1& \alpha& \alpha& \alpha+1& \alpha+1\\ 
1& \alpha& \alpha& \alpha+1& \alpha
\end{array}
\right )
\quad \text{and} \quad  
\left (
\begin{array}{ccccc}
\alpha+1 & \alpha+1 & \alpha +1 & 1 & 1 \\ 
0 & 0 & 1 & 0 & 1\\ 
\alpha & \alpha+1 & \alpha+1 & 1 & 1\\  
0 & 1 & 0 & 1 &  0
\end{array}
\right ).
\]
\end{example}
\medskip

%%%%%%%%%%%%%%%%%%%%
%%%%%%%%%%%%%%%%%%%%
%%%%%%%%%%%%%%%%%%%%
%%%%%%%%%%%%%%%%%%%%
%%%%%%%%%%%%%%%%%%%%
%%%%%%%%%%%%%%%%%%%%
%%%%%%%%%%%%%%%%%%%%
%%%%%%%%%%%%%%%%%%%%
%%%%%%%%%%%%%%%%%%%%
%%%%%%%%%%%%%%%%%%%%
%%%%%%%%%%%%%%%%%%%%

\section{Codes over $\TT(w_1,w_2)$}\label{sec: codes in 1-dim}

In this section we study the weighted projective Reed--Muller codes over a $1$-dimensional torus $\TT(w_1,w_2)$. In this case $I_\TT$ is always a complete intersection 
(\emph{cf.}~Remark~\ref{rmk: when do we get a complete intersection}):
 $$I_\TT=\bigl(t_1^{(q-1)w_2}-t_2^{(q-1)w_1}\bigr).$$ 
By a classical result of Sylvester, the Frobenius number of $Q=\Span{w_1,w_2}$ is $\mathsf{g}_Q=w_1w_2-w_1-w_2$. According to Corollary~\ref{cor: regularity}, the index of 
regularity of $K[t_1,\dots,t_s]/I_\TT$ is $(q-1)w_1w_2-w_1-w_2+1$. Hence, we restrict to the range $1\leq d \leq (q-1)w_1w_1 - w_1 -w_2$. 
We will show below in Corollary~\ref{cor: MDS codes} that all weighted projective Reed--Muller codes over a $1$-dimensional (weighted) torus
are maximum distance separable codes.

Given a semigroup $Q\subset \NN$, let us denote by 
$\chi_Q\colon \NN \rt \set{0,1}$ the characteristic function of $Q\subset \NN$, \emph{i.e.},
the function given by $\chi_Q(d)=1$ if and only if $d\in Q$ and $\chi_Q(d)=0$ otherwise. 
We use this function for the semigroup $Q=\Span{w_1,w_2}$ only; to ease notation we will write simply $\chi$.

\begin{proposition}\label{prop: dimension of the code over the torus}
Let $0\leq d \leq w_1w_2(q-1)-w_1-w_2$. 
Write $d = kw_1w_2 + l$, where $k\geq 0$ and $0\leq l < w_1w_2$. Then, 
 $\dim_K C_\TT(d)=k+\chi(l)$.
\end{proposition}

\begin{proof} 
The Hilbert series of $I_\TT$ is
\[
H_\TT (t) = \frac{1- t^{(q-1)w_1w_2}}{(1-t^{w_1})(1-t^{w_2})} = \bigl(1+ t^{w_1} + \cdots + t^{(q-1)w_1w_2-w_1}  \bigr)\bigl(1+t^{w_2}  + t^{2w_2} +\cdots \bigr).
\]
Hence, the dimension of $C_\TT(d)$ coincides with the coefficient of the monomial in $t^d$ on the right hand side of the above equation. 
Suppose that $a,b\in \NN$ are such that $d=aw_1+bw_2$. Then $aw_1\leq aw_1+bw_2 < (q-1)w_1w_2$ implies that 
$a\leq (q-1)w_2-1$. Hence the coefficient of $t^d$ on the right hand side of the equation, is 
the number of pairs $(a,b)\in \NN^2$ such that $d=aw_1+bw_2$, or, in the context of numerical semigroups, the number of \emph{factorizations}
of $d$ in $Q=\Span{w_1,w_2}$. 
\smallskip

Let us compute this number. Since $\mathsf{g}_Q=w_1w_2-w_1-w_2$, we see that $l+(1-\chi(l))w_1w_2\in Q$. Let 
$a,b\in \NN$ be such that $l+(1-\chi(l))w_1w_2 = aw_1 + bw_2$. Then
\[
d = kw_1w_2 +l =  (a+iw_2)w_1+(b+(k-1+\chi(l)-i)w_1)w_2
\]
for $i=0,\dots,k-1+\chi(l)$, yields $k+\chi(l)$ distinct factorizations of $d$. Consider $\set{a_iw_1+b_iw_2}_{i=1}^r$ 
the set of all factorizations of $d$. We may assume $a_r>a_{r-1}>\cdots >a_1\geq 0$. 
Since the difference $a_i-a_{i-1}$ must be divisible by $w_2$ we get $a_r \geq (r-1)w_2$. Therefore 
\begin{equation}\label{eq: will give l in Q}
d-(r-1)w_1w_2 = (a_r-(r-1)w_2)w_1 + b_rw_2\in Q=\Span{w_1,w_2}.
\end{equation}
Additionally,  $(r-1)w_1w_2\leq a_rw_1\leq d = kw_1w_2+l\leq kw_1w_2$, hence $r\leq k +1$. Now, 
if $r=k+1$, then, by (\ref{eq: will give l in Q}), $l\in Q$. This shows that 
$r\leq k + \chi(l)$.
\end{proof}

Let us denote by $\alpha\in K^*$ a choice of generator 
of the cyclic group $K^*$. Given a homogeneous $f\in K[t_1,\dots,t_s]$, we denote by $V(f)$ its set of zeros in $\PP(w_1,\dots,w_s)$.

\begin{lemma}\label{lemma: vanishing ideal of a point}
For each $0\leq r\leq q-2$, 
\mbox{$V(t_1^{w_2}-\alpha^r t_2^{w_1})\subset \TT(w_1,w_2)$} consists of a single point. Moreover,
as $r$ varies in $\set{0,\dots,q-2}$, every point of $\TT(w_1,w_2)$ is obtained in this way. 
\end{lemma}

\begin{proof} Fix $a,b\in \ZZ$ such that $aw_1+bw_2=1$.
Clearly $(\alpha^{rb},\alpha^{-ra})\in V(t_1^{w_2}-\alpha^r t_2^{w_1})$. Suppose $(x_1,x_2)\in \TT(w_1,w_2)$ belongs to $V(t_1^{w_2}-\alpha^r t_2^{w_1})$, \emph{i.e.},
$x_1^{w_2}=\alpha^{r} x_2^{w_1}$. Then:
\[
(x_1,x_2)=(x_1(x_1^{-a}x_2^{-b})^{w_1}, x_2(x_1^{-a}x_2^{-b})^{w_2})= (x_1^{bw_2}x_2^{-bw_1},x_1^{-aw_2}x_2^{aw_1})=(\alpha^{rb},\alpha^{-ra}). 
\]
Hence $V(t_1^{w_2}-\alpha^r t_2^{w_1})=\set{(\alpha^{rb},\alpha^{-ra})}$. To show that every point in $(x_1,x_2)\in \TT(w_1,w_2)$ is the zero of one such 
polynomial it suffices to notice that $x_1^{w_2}/x_2^{w_1}=\alpha^r$, for some $0\leq r \leq q-2$.
\end{proof}

\begin{proposition}\label{prop: max zeros of a degree d poly}
Let $f\in K[t_1,t_2]$ be nonzero, homogeneous of degree $d$. Write $d = kw_1w_2 + l$, where $k\geq 0$ and $0\leq l < w_1w_2$. 
Then $\ls{V(f)\cap \TT(w_1,w_2)}\leq k -1+ \chi(l)$.
\end{proposition}

\begin{proof}
We argue by induction on $k$. 
Suppose that $k = 0$. Then $d<w_1w_2$ and, by an argument similar to the one used in the proof of Proposition~\ref{prop: dimension of the code over the torus},
we deduce that there is only one factori\-zation of $d$ in $Q$, hence $f$ is a monomial and thus $\ls{V(f)\cap \TT(w_1,w_2)} = 0$. 
Additionally, if $k=0$ then $l=d\in Q$ and so $\chi(l)=1$ and the inequality of the statement holds.
\smallskip

Suppose $k \geq 1$. Let us write
\mbox{$f=g t_1^at_2^b$}, for some $g\in K[t_1,t_2]$, such that neither $t_1$ nor $t_2$ divides $g$.
Let $d'=d-aw_1-bw_2$, be the degree of $g$. If $d'<w_1w_2$ then $g=1$. In this situation $V(f)\cap \TT(w_1,w_2)$ is empty and there is nothing to show.
Suppose $d'\geq w_1w_2$. Let us write $d'=k'w_1w_2+l'$ for $1\leq k'\leq k$ and $0\leq l'<w_1w_2$. If $k'<k$, by induction we get:
\begin{equation}\label{eq: inequality for g}
\renewcommand{\arraystretch}{1.3}
\begin{array}{l}
\ls{V(f)\cap \TT(w_1,w_2)}  =  \ls{V(g)\cap \TT(w_1,w_2)}  \leq  
k' -1+ \chi(l')\leq k -1 + \chi(l).
\end{array} 
\end{equation}
If $k=k'$ then $l = l' +aw_1+bw_2$ and thus $\chi(l')\leq \chi (l)$.
We may assume there exists $(x_1,x_2)\in V(f)\cap \TT(w_1,w_2)$.
Let us write $g=\sum_{i=0}^r \alpha_i t_1^{a_i}t_2^{b_i}$, with $r\geq 1$ and, without loss in generality, $0=a_0\leq a_1\leq \cdots \leq a_r$. 
Since $w_1a_i +w_2b_i = d' = w_2b_0$, we deduce
that there exist $m_i\geq 0$ such that $a_i = m_iw_2$ and 
$b_i = b_0 - m_iw_1$. Hence, we may write in $\operatorname{Frac}K[t_1,t_2]$:
\[
g = t_2^{b_0} \sum_{i=0}^r \alpha_i\left (\frac{t_1^{w_2}}{t_2^{w_1}}\right )^{m_i} = t_2^{b_0} G\bigl(t_1^{w_2}/t_2^{w_1} \bigr),
\]
where $G(z)=\sum_{i=0}^r \alpha_i z^{m_i}\in K[z]$ has degree $m_r=a_r/w_2$.
We see that $x_1^{w_2}/x_2^{w_1}$ is a zero of $G$. Let $0\leq r \leq q-2$ be such that $\alpha^r = x_1^{w_2}/x_1^{w_1}$. Then \mbox{$G(z)=H(z)(z-\alpha^r)$}, for some $H\in K[z]$, of degree
$a_r/w_2-1$. Accordingly, $g =  t_2^{b_0} H\bigl(t_1^{w_2}/t_2^{w_1} \bigr)\bigl(t_1^{w_2}/t_2^{w_1} - \alpha^r \bigr)$.
Since $$b_0-w_1 \geq w_1(a_r/w_2-1) \iff b_0w_2 \geq w_1a_r \iff d'\geq w_1a_r,$$ clearing denominators, we conclude that there exists $h\in K[t_1,t_2]$, homogeneous, such that 
$g = (t_1^{w_2}- \alpha^r t_2^{w_1})h$. Since the degree of $h$ is $(k'-1)w_1w_2 + l'$, by induction and Lemma~\ref{lemma: vanishing ideal of a point} 
\[
\ls{V(f)\cap \TT(w_1,w_2)}  =  \ls{V(h)\cap \TT(w_1,w_2)} + 1 \leq  
(k'-1)-1+ \chi(l') + 1\leq k -1 + \chi(l). \qed
\]
\noqed
\end{proof}
\vspace{-.75cm}

We now address the computation of the minimum distance of the weighted projective Reed--Muller codes over
a weighted torus. Recall that the minimum distance is defined for a nonzero code. For this, the assumption that $d\in Q$, equivalent to 
$\dim_K C_\TT(d) \not = 0$, is necessary in the statement of the theorem.

\begin{theorem}\label{thm: computing the minimum distance}
If $0\leq d \leq   w_1w_2(q-1)-w_1-w_2$ and $d\in Q$. Write $d = kw_1w_2 + l$ with $k\geq 0$ and $0\leq l<w_1w_2$. Then the minimum distance 
of the evaluation code $C_\TT(d)$ is $(q-1)-k+1-\chi(l)$.
\end{theorem}

\begin{proof}
Let $f\in K[t_1,t_2]$ be a homogeneous polynomial of degree $d$. 
Then, by Proposition~\ref{prop: max zeros of a degree d poly},
$f$ has at most $k -1 + \chi(l)$ 
zeros on $T(w_1,w_2)$. Since, by Proposition~\ref{prop: length of the weighted torus}, the length of
$C_\TT(d)$ is $q-1$ we get $\delta_\TT(d)\geq (q-1)-k +1 - \chi(l)$.
To prove the reverse inequality, we split the proof into $2$ cases. If 
$\chi(l)=1$, let $a,b\in \NN$ be such that 
$l= aw_1+bw_2$. Then, the polynomial 
$$\textstyle f = t_1^at_2^b\prod_{i=1}^{k} (t_1^{w_2} - \alpha^i t_2^{w_1})$$
has degree $d$ and, since $d\leq w_1w_2(q-1)-w_1-w_2$ implies that $0\leq k\leq q-2$, by Lemma~\ref{lemma: vanishing ideal of a point},
has exactly $k=k-1+\chi(l)$ zeros on $\TT(w_1,w_2)$. 
If $\chi(l)= 0$ then, since $d=kw_1w_2 + l \in Q$ we must have $k>0$. Additionally, since \mbox{$d-w_1w_2(k-1)\geq w_1w_2> \mathsf{g}_Q$}, there exist
$a,b\in \NN$ such that $aw_1+bw_2 = d-w_1w_2(k-1)$. Then the polynomial \mbox{$f = t_1^at_2^b\prod_{i=1}^{k-1} (t_1^{w_2} - \alpha^i t_2^{w_1})$}
has degree $d$ and has exactly $k-1=k-1+\chi(l)$ zeros
on $\TT(w_1,w_2)$.
\end{proof}

\begin{corollary}\label{cor: MDS codes}
The weighted projective Reed--Muller codes $C_\TT(d)$ over a (weighted) torus are 
maxi\-mum distance separable codes, i.e., for $d\in Q$, $\delta_\TT(d)=(q-1)-\dim C_\TT(d) +1$.
\end{corollary}

\bibliographystyle{plain}

\end{document}